\documentclass[12pt]{amsart}
\usepackage{amsfonts}
\numberwithin{equation}{section}
\usepackage{amssymb}
\usepackage{amsmath}

\def\Mod{\operatorname{mod}}
\def\intt{\operatorname{int}}
\def\Id{\operatorname{Id}}
\newtheorem{lemma}{Lemma}[section]
\newtheorem{theorem}{Theorem}

\newtheorem{cor}{Corollary}

\newtheorem{Remark}{Remark}[section]
\newtheorem{Prop}{Proposition}[section]
\newtheorem{Question}{Question}

\title[Julia set]{A new geometric characterization of the Julia set}
\author{Xiao Yao}
\address{Department of Mathematical Sciences,
Tsinghua University, Beijing, 100084, People's Republic of China}
\email{yaoxiao0710107@163.com}

\author{Daochun Sun}
\address{School of Mathematical Sciences,
China Normal University, Guangzhou, 510631, People's Republic of China}
\email{1457330943@qq.com}

\author{Zuxing Xuan}
\address{Beijing Key Laboratory of Information Service Engineering, Institute of Applied Sciences, Department of General Education, Beijing Union University, Beijing, 100101, People's Republic of China}

\email{zuxingxuan@163.com}
\begin{document}

\begin{abstract}
 This article concerns a new geometric
characterization of the Julia set. By using Ahlfors-Shimizu's characteristic, we establish some growth results which indicates the characterization of the Julia set.
The main technique is to estimate  the lower bound of $S(f^n,U)$, where $U$ is an open neighbourhood of some point in $\mathcal{J}(f)$.
\end{abstract}

\maketitle

\section{Introduction}

Complex dynamic system concerns the iteration of a holomorphic map on the simply connected Riemann
surface $S$. By the uniformization theorem, $S$ must be
conformally equivalent to $\Delta:=\{z: |z|<1\}$, $\mathbb{C}$, or $\mathbb{\hat{C}}:=\mathbb{C}\cup\{\infty\}$.  In this paper, we only restrict ourselves to the case that $S$ is bi-holomorphic to $\mathbb{C}$, or $\hat{\mathbb{C}}.$ Given a point $z\in S$, we consider the sequence generated by its
iterates $\{f^n(z)\}_{n=1}^{+\infty}$ and study the possible behaviors as $n$ tends to
infinity. We divide $S$ into the Fatou set
$$\mathcal{F}(f)=\{z\in S:\{f^n\}_{n\in \mathbb{N}}\ \text{is a normal family in some neighbourhood of } z\}$$
and the Julia set $\mathcal{J}(f):= S \backslash \mathcal{F}(f)$. We will consider three cases:
\begin{enumerate}
\item
rational function: $f:\mathbb{\hat{C}}\longrightarrow \mathbb{\hat{C}}$;
\item transcendental entire function: $f:\mathbb{C}\longrightarrow \mathbb{C}$;
\item transcendental meromorphic function: $f:\mathbb{C}\longrightarrow \mathbb{\hat{C}}$.
\end{enumerate}
The iteration of rational functions was initiated by
Julia \cite{Julia-Rational} and Fatou \cite{Fatou-Rational}
 between 1918 and 1920, and the iteration of transcendental
 entire functions was first considered by
 Fatou \cite{Fatou-entire} in 1926. For the general
 theory of iteration mermorphic functions, we refer
 to Bergweiler's survey \cite{Bergweiler-Survey}.
 For the detailed discussions for dynamical  properties
 of Fatou and Julia sets, please see
 \cite{Beardon, Milnor-2006Book, Eremenko-Lyubich,  Eremenko-Lyubich-Survey,Sch}.

 The expansion property near the Julia set is common for rational maps, transcendental entire (mermorphic) maps.  To our best knowledge, there are two main classical definitions of the Julia set. One is based on the normal family argument, and the other is based on the fact that $\mathcal{J}(f)$ is the closure of the set of repelling periodic points.  Sun and Yang investigated  the dynamics of a more general class of maps \cite{SunYang2000}, which are called quasi-rational maps. They explored the idea of Ahlfors covering surface theory to give  new definitions of Fatou set and Julia set. This paper aims to give a new geometric characterization of the Julia set by using the spherical metric. Let $f$ be an analytic function on a domain $V\subset \mathbb{C}$.  For any subset $U$ of $V$, we define
$$S(f, U):=\frac{1}{\pi}\int\int_{U}\frac{|f'(z)|^2}
{(1+|f(z)|^2)^2}dx dy,\ \text{where}\ \ z=x+iy.$$
Throughout this paper, we use $D(z, r)=\{w: |w-z|< r\}$
for $z\in \mathbb{C}$. In the sequence, we  always
assume that a transcendental meromorphic function has
at least one pole.

 We begin with our characterization of the Julia set by the following.

\begin{theorem}\label{characterization-general}
 Let $f$ be a rational function on $\mathbb{\hat{C}}$ with degree $d\geq 2$, or a transcendental entire  (meromorphic) function $f$ on $\mathbb{C}$. Then
 $z\in \mathcal{J}(f)$ if and only
\[
\liminf\limits_{n\rightarrow +\infty}S(f^n, D(z,r))=+\infty
\]
  for any $r>0$.

 Especially, if $f(z)$ is a rational function  on $\hat{\mathbb{C}}$ with degree $d\geq 2$. We have
$z\in \mathcal{J}(f)$ if and only
\begin{equation}\label{Characterization-rational}
\lim\limits_{n\rightarrow +\infty}\frac{\log
S(f^n, \ D(z,r))}{n}=\log d
\end{equation}
 for any $r>0$.

\end{theorem}
It is important to choose a suitable metric on a neighborhood of the Julia set, especially for transcendental entire (meromorphic) function $f$ for which $\mathcal{J}(f)$ is not compact. The hyperbolic metric is widely used in complex dynamical system. We refer to \cite{Ahlfors-conformal-invariant, Beardon-Minda} for the detailed introduction for the hyperbolic metric and \cite{ Eremenko-Lyubich, Mcmullen-Renomalization, Stallard-hyperbolic, Rempe-Sixsmith} for the application of hyperbolic metric in dynamic systems.

 Here, we emphasize that the spherical metric is crucial to Theorem \ref{characterization-general}.  We remark that it seems that Euclidean metric is not suitable for the characterization, and we only mention that the existence of Baker wandering domains is the obstruction for the characterization. For any transcendental entire function $f$ with a multiply connected Fatou component $U$,  we can choose  $z_0\in U$,
$$\lim_{n\rightarrow +\infty}\mathcal{L}^2(f^n, D(z_0, r))=+\infty$$
for any disk $D(z_0,r)\subset U$, where $\mathcal{L}^2 (.)$
means the two dimensional Lebesgue measure.
For the details, see the remarks and discussions in
Proposition \ref{effective} in Section 3.
 Motivated by equation (\ref{Characterization-rational}) for the rational case, we ask the following question for transcendental case.
\begin{Question}
Let $f$ be a transcendental entire (mermorphic) function.
Then for  any   $z\in \mathcal{J}(f)$, is it true that
$$\lim\limits_{n\rightarrow +\infty}\frac{1}{n}\log S(f^n, \ D(z,r))=+\infty$$ for any $r>0$?
\end{Question}

  Suppose that  $f$ is a transcendental meromorphic function $f$ with at least two poles. The fact that the  pre-poles are dense in $\mathcal{J}(f)$ implies that $S(f^n, D(z,r))=+\infty$ for some $n$.
For other cases, things become somewhat subtle. One can not
obtain the uniform estimation due to the lake of the compactness
of the Julia set. We conjecture that the answer is affirmative.
 However, it seems  difficult to answer it completely. We will give some partial results.

\begin{theorem}\label{characterization-one}
Let $f$ be a transcendental entire function. Assume that
\begin{equation}\label{Growth-Condtion}
\liminf_{r\rightarrow+\infty}\frac{\log M(2r, f)}{\log M(r,f)}=d>2,
\end{equation}
then $z\in \mathcal{J}(f)$ if and only if
\[
\lim_{n\rightarrow+\infty}\frac{1}{n}\log S(f^n, D(z,r))=+\infty
\]
for any $r>0$, where $M(r,f)=\max\limits_{|z|=r}|f(z)|$.
\end{theorem}

In fact, we can weaken the growth condition \eqref{Growth-Condtion} in Theorem \ref{characterization-one} to  get the following.

\begin{cor}\label{characterization-k}
Let $f$ be a transcendental entire function. Assume that
\begin{equation}\label{Growth-Condtion-k}
\liminf_{r\rightarrow+\infty}\frac{\log M(2r, f^k)}
{\log M(r,f^k)}=d_k>2
\end{equation}
for some $k\in \mathbb{N}$,
then $z\in \mathcal{J}(f)$ if and only if
\[
\lim_{n\rightarrow+\infty}\frac{\log S(f^n, D(z,r))}{n}=+\infty
\]
for any $r>0$.

\end{cor}
We remark that the exponential family and cosin family  correspond to the case $k=2$ in Corollary \ref{characterization-k}. The condition \eqref{Growth-Condtion} can be weakened if the exceptional set $E(f)=\{z\in \hat{\mathbb{C}}: \sharp(\mathcal{O}^{-}(z))<\infty\}$ is non-empty, where $\mathcal{O}^{-}(z)=\{f^{-n}(z)\}_{n=0}^{\infty}$.  Due to Picard's theorem, we know that
$\sharp{(E(f)\backslash\{\infty\})}\leq 1$ for any transcendental entire function.
\begin{theorem}\label{characterization-E(f)-empty}
Let $f$ be a transcendental entire function $f$. Assume that $E(f)=\emptyset$ and
\begin{equation}\label{Growth-Condtion-2}
\liminf_{r\rightarrow+\infty}\frac{\log M(2r, f)}{\log M(r,f)}=d>1,
\end{equation}
then $z\in \mathcal{J}(f)$ if and only if
\[
\lim_{n\rightarrow+\infty}\frac{1}{n}\log S(f^n, D(z,r))=+\infty
\]
for any $r>0$.
\end{theorem}
This theorem can also lead to a corollary which is very similar to Corollary \ref{characterization-k}, and we omit its statement here. One can use growth condition \eqref{Growth-Condtion} or
\eqref{Growth-Condtion-2}
to  rule out the existence of multiply connected Fatou components by the following result by Zheng.

\noindent {\bf Theorem A.}(\cite{Zheng2006})
Let $f$ be a transcendental entire function. Assume that
\begin{equation}\label{Zheng-Condition}
\liminf_{r\rightarrow+\infty}\frac{\log M(2r, f)}{\log M(r,f)}>1,
\end{equation} then
$\mathcal{F}(f)$ has no multiply connected Fatou component.

\begin{Remark}
 The original proof of Theorem A makes use of Nevanlinna theory and a covering lemma \cite{Zheng2006}, which also works for transcendental meromorphic functions with finitely many poles. The idea in the proof of Theorem \ref{characterization-one} can give an intuitional explanation of  Theorem A under the growth condition \eqref{Growth-Condtion}, which is stronger than the condition \eqref{Zheng-Condition}.  The existence of multiply connected Fatou components implies that one can choose a large annuli $U_{r}=\{z: r< |z|< 2^k r\}$ for some large $r\in \mathbb{R}$ and $k\in \mathbb{N}$ in some multiply connected Fatou component. Then $\sup\limits_{n}S(f^n, U_{r})=+\infty$ follows from the similar technique as in the proof of Theorem \ref{characterization-one}. Since $U_{r}$ lies in a bounded Fatou component, we have $\sup\limits_{n}S(f^n, U_{r})<+\infty$ by Theorem \ref{characterization-general}, which is a contradiction.
\end{Remark}
\begin{Remark}
By Theorem A, if a transcendental entire function $f$
satisfies the growth condition $\eqref{Growth-Condtion-k}$ in
Corollary $\ref{characterization-k}$, then
$\mathcal{J}(f)$ is connected. To our best knowledge,
we do not realize that there exists a
transcendental entire function $f$ such
that $\mathcal{J}(f)$ is connected and $f$ satisfies
the counterpart of growth
condition \eqref{Growth-Condtion-k},
\[
\liminf_{r\rightarrow+\infty}
\frac{\log M(2r, f^k)}{\log M(r,f^k)}=d_k\in [1,2]
\]
for each $k$. This is related to an equivalent form of Theorem A, which can be re-stated as the follows.

\noindent {\bf Theorem B.}
Let $f$ be a transcendental entire function. If $\mathcal{J}(f)$
 is disconnected, then
\[
\liminf_{r\rightarrow+\infty}\frac{\log M(2r, f^k)}
{\log M(r,f^k)}=1
\]
for each $k$.

Motivated by Theorem \ref{characterization-one}, Corollary
\ref{characterization-k} and Theorem B, we post the following
questions.

\begin{Question}[Strong version]\label{Question-Strong-Version}
Does there exist a transcendental entire function $f$ such that
$\mathcal{J}(f)$ is connected and
\[\liminf_{r\rightarrow+\infty}\frac{\log M(2r, f^k)}
{\log M(r,f^k)}=1\]
for each $k\in \mathbb{N}$?
\end{Question}

\begin{Question}[Weak version]\label{Question-Weak-version}
Does there exist a transcendental entire function $f$ such that
$\mathcal{J}(f)$ is connected and
\[\liminf_{r\rightarrow+\infty}\frac{\log M(2r, f^k)}
{\log M(r,f^k)}=d_k\in [1,2]\]
for each $k\in \mathbb{N}$?
\end{Question}

 Question \ref{Question-Strong-Version} and Question
\ref{Question-Weak-version} are asked for the general transcendental
entire function $f$. In fact, we could further restrict $f$ into  some special
class of functions  to consider the  questions above. One can still replace $\log M(r,f)$ by the Nevanlinna's
characteristic $T(r,f)$ or Ahlfors-Shimizu's characteristic $T_{0}(r,f)$ to discuss the similar questions. For the definitions of these two
characteristics, we refer to Section \ref{Preliminaries}.
\end{Remark}
Up to now,  we investigate the case that $f$ has only simply connected Fatou components. A nature question is whether we can get the similar results when $f$ has some multiply connected Fatou component $U$?  We will give a partial answer under the condition $e(U)=2$, where $e(U)$ denotes the eventual connectivity of $U$. Kisaka and Shishikura \cite{Kisaka-Shishikura} proved that $e(U)$ can only be $2$ or $\infty$ for any multiply connected Fatou component of a transcendental entire function.  In the sequence, we will consider the case that $f$ has a multiply connected Fatou component. We confine the case that the eventual connectivity is $2$ and  get the following.

\begin{theorem}\label{Multiply}
Let $f$ be a transcendental entire function. Assume  that   $\mathcal{F}(f)$ has a multiply Fatou component with eventual connectivity $2$. Then $z\in \mathcal{J}(f)$ if and only if
\[
\lim_{n\rightarrow+\infty}\frac{1}{n}\log S(f^n, D(z,r))=+\infty
\]
for any $r>0.$
\end{theorem}

We believe that the growth condition \eqref{Growth-Condtion} can be completely removed. Here, we prove a theorem without any growth condition, while we need to restrict ourselves in hyperbolic transcendental entire functions in class
$$ \mathcal{B}=\{f: f \ \text{is a transcendental entire function}, S(f) \ \text{is bounded}\}, $$
 where $S(f)$ is the smallest closed set such that
 $f: \mathbb{C}\backslash f^{-1}S(f)\rightarrow \mathbb{C}
 \backslash S(f)$ is a covering map. Let $f$ be a transcendental entire function in class $\mathcal{B}$, $f$ is said to be hyperbolic if $\mathcal{F}(f)$ consists of finitely many attracting basins. We refer to the recent article \cite{Rempe-Sixsmith} and the references therein.

\begin{theorem}\label{characterization-hyperbolic}
Let $f$ be a hyperbolic transcendental entire function in class $\mathcal{B}$. Assume that  $E(f)\subset \mathcal{F}(f)$, then
 $z\in \mathcal{J}(f)$ if and only if
\[
\lim_{n\rightarrow+\infty}\frac{1}{n}\log S(f^n, D(z,r))=+\infty
\]
for any $r>0.$
\end{theorem}

We also include the case  $E(f)=\emptyset$  in Theorem \ref{characterization-hyperbolic}.  It can be extended to some larger class of transcendental maps, while we do not pursue the full generality here.

The paper consists of five sections. In section \ref{Preliminaries}, we collect together a number of results that will be used later and introduce some notations and definitions. In section \ref{Section-Theorem-1}, we will prove Theorem \ref{characterization-general}. In section \ref{Section-Theorem-2-3}, the proofs of Theorem \ref{characterization-one} and Theorem \ref{characterization-E(f)-empty} are given. In section \ref{Section-Theorem-4-5}, we prove Theorem \ref{Multiply} and Theorem \ref{characterization-hyperbolic}.

\section{Preliminaries}\label{Preliminaries}

In this section, we first give some  definitions and basic facts. We will give  a brief introduction of Nevanlinna theory (See Hayman's book \cite{Hayman-Book}). Let $f$ be a transcendental entire (meromorphic)  function. Set
$\log^{+}x=\log\max\{1, x\}$.  Define
$$m(r,f):=\frac{1}{2\pi}\int_{0}^{2\pi}\log^{+}|f(re^{i\theta})|d\theta,$$
$$N(r,f):=\int_{0}^{r}\frac{n(t,f)-n(0,f)}{t}dt+n(0,f)\log r,$$
where $n(t,f)$ denotes the number of poles of $f$ counted according to their multiplicities
in $\{z:|z|<t\}$, and
$$T(r,f):=m(r,f)+N(r,f).$$
$N(r,f)$ is known as the integrated counting function of poles of $f$ and $T(r,f)$ as the Nevanlinna's characteristic of $f$. Then $f$ is transcendental if and only if $$\lim_{r\rightarrow+\infty}\frac{T(r,f)}{\log r}=+\infty.$$
The Nevanlinna deficiency is defined as
$$\delta(a,f)=\liminf_{r\rightarrow+\infty}\frac{m(r,\frac{1}{f-a})}{T(r,f)}.$$

Given a domain $U$ in $\mathbb{C}$, recall that
$$S(f, U):=\frac{1}{\pi}\int\int_{U}\frac{|f'(z)|^2}
{(1+|f(z)|^2)^2}dx dy.$$ In order to be compatible with the conventions in Nevanlinna theory, we write $S(r,f)=S(f, D(0,r))$ for simplicity, and we expect this  will bring no confusion. The function
$$T_0(r,f)=\int_0^{r}\frac{S(t,f)}{t}dt$$
is called the Ahlfors-Shimizu's characteristic of $f$.

\begin{lemma}(\cite{Hayman-Book})\label{lemma-Nevanlinna-Ahlfors-Shimizu}
Let $f$ be a transcendental entire function. Then
$$|T_0(r,f)-T(r,f)-\log^{+}|f(0)||\leq \frac{\log 2}{2}.$$
\end{lemma}

\begin{lemma}(\cite{Hayman-Book})\label{Modulus-Nevanlinna-Characteristic}
Let $f$ be a transcendental entire function. Then
$$T(r,f)\leq \log^{+}M(r,f)\leq \frac{R+r}{R-r}T(r,f)$$
for any $r<R$.
\end{lemma}
We need to estimate $S(f^n, U)$ carefully  when $n$ varies.  By a direct calculation, we  get the following proposition, which will be  used  directly  in the sequence.
\begin{Prop}\label{Area-Relation}
Let $f$ be a rational function on $\hat{\mathbb{C}}$ or a transcendental entire function on $\mathbb{C}$. Assume that $U$, $V$ be two open sets in $\mathbb{C}$ such that $V\subset f^{k}(U)$ for some $k\in \mathbb{N}$, then
$$S(f^{n+k},U)\geq S(f^n, V)$$ for any
$n\in \mathbb{N}$.
\end{Prop}
\begin{proof}
The proof  just follows from a direct computation.
\begin{align*}
S(f^{n+k}, U)&=\frac{1}{\pi}\int_{U}\frac{|(f^{n+k})'(z)|^2}{(1+|f^{n+k}(z)|^2)^2}dxdy\\
&=\frac{1}{\pi}\int_{U}\frac{|(f^{n})'(f^{k}(z)|^2|(f^k)'(z)|^2}
{(1+|f^n(f^k (z))|^2)^2}dxdy\\
&\geq \frac{1}{\pi}\int_{f^{k}(U)}\frac{|(f^{n})'(z)|^2}{(1+|f^n(z) |^2)^2}dxdy\\
 &\geq \frac{1}{\pi}\int_{V}\frac{|(f^{n})'(z)|^2}{(1+|f^n(z) |^2)^2}dxdy\\
 &=S(f^n, V).
\end{align*}

\end{proof}

\section{The proof of Theorem \ref{characterization-general}}\label{Section-Theorem-1}
\subsection{Sufficient part of Theorem \ref{characterization-general}}\

In this section, we will prove  the sufficient part of Theorem 1.  The following proposition is established by Marty's Criterion.
\begin{Prop}\label{Characterization-Fatou-Set}
Let $f$ be a rational function on $\mathbb{\hat{C}}$, or a transcendental entire  (meromorphic ) function $f$ on $\mathbb{C}$. $\forall z\in \mathcal{F}(f)$, we have
  $$\limsup_{n\rightarrow+\infty}S(f^n, D(z,r))<+\infty$$ for some $r>0$.
\end{Prop}

\begin{proof}
For any $z_0\in \mathcal{F}(f)$, by Marty's Criterion, there exists $r>0$ such that
\[\sup_{n}\sup_{z\in D(z_0,r)}\frac{|(f^n)'(z)|}{1+|f^n(z)|^2}\leq \Gamma\]
for some constant $\Gamma$.

Now we obtain
\begin{align*}
S(f^{n}, D(z_0,r))=\frac{1}{\pi}\int\int_{D(z_0,r)}\frac{|(f^{n})'(z)|^2}
{(1+|f^n(z)|^2)^2}dxdy\leq \Gamma^2 r^2.
\end{align*}
\end{proof}
\begin{Remark}
Here we point out the above proof also works for the sufficient part of Theorems 2-4, we omit them in the sequence. We just need to prove the necessary part of these theorems.
\end{Remark}

As we remarked before, we know that Euclidean metric can not be used in the characterization of the Julia set of transcendental functions in general.  This proposition essentially uses the Marty's Criterion, that is  why we  choose spherical metric. Also it is interesting to consider some other metrics to give the characterization of the Julia set.   Let $\rho(z) |dz| $ be a conformal metric on $\mathbb{C}$, one can define the $\rho$-area of $U$ as
$$S_{\rho}(f, U):=\int_{U}\rho^2(f(z)) |f'(z)|^2dxdy$$
for any $f\in \mathcal{H}$, where $\mathcal{H}$ denotes some subclass of rational (meromorphic, entire) functions. We call $\rho(z)|dz|$ is effective for  $\mathcal{H}$, if for any $f\in \mathcal{H}$
$$z\in \mathcal{J}(f)\Longleftrightarrow
  \lim\limits_{n\rightarrow+\infty}S_{\rho}(f^n, D(z,r))=+\infty$$ for any $r>0$.

In the following, we discuss the effectiveness of the Euclidean metric.
\begin{Prop}\label{effective}
Let $\rho(z)=|dz|$ be the Euclidean metric.  We have

(1) $\rho(z)$ is not effective for rational functions;

(2) $\rho(z)$ is effective for $\mathcal{S}$, where
$$ \mathcal{S}=\{f: f \ \text{is a transcendental entire function}, S(f) \ \text{is finite}\}, $$
 and $S(f)$ is the smallest closed set such that
 $f: \mathbb{C}\backslash f^{-1}S(f)\rightarrow \mathbb{C}
 \backslash S(f)$ is a covering map.
\end{Prop}
\begin{proof}
(1) We consider  the example of $f(z)=z^2$.
For any fixed $z$ such that $|z|>1$, then for any $r<\frac{|z|-1}{2}$, we have
$$\liminf_{n\rightarrow+\infty}S_{\rho}(f^n, D(z,r))=\liminf_{n\rightarrow+\infty}
\int_{D(z,r)}|(f^n)'(z)|^2dxdy=+\infty.$$
This indicates that the Euclidean metric is not effective for rational functions.

(2) For any $z\in \mathcal{J}(f)$, the expansion property near the Julia set implies that
$$\liminf\limits_{n\rightarrow+\infty}S_{\rho}(f^n, D(z,r))=+\infty$$ for any $r>0$.
For any $z\in \mathcal{F}(f)$, by the classification of Fatou components \cite{Sch}, we know that $f^{n}(z)$ will eventually enter into a periodic attracting cycle, a periodic parabolic cycle or a periodic cycle of Siegel disks. There exists $r>0$ such that $\{f^n(D(z,r))\}_{n\in \mathbb{N}}$ will stay in some bounded region. We know that in a bounded region, the spherical metric $\hat{\rho}(z)=\frac{1}{\pi}\frac{|dz|}{1+|z|^2}$ and Euclidean metric $\rho(z)=|dz|$ are equivalent to each other. By Theorem \ref{characterization-general}, we know that $\lim\limits_{n\rightarrow+\infty}S_{\hat{\rho}}(f^n, D(z,r))<+\infty$, which implies that $\lim\limits_{n\rightarrow +\infty}S_{\rho}(f^n, D(z,r))<+\infty.$

\end{proof}

We split the proof of Theorem \ref{characterization-general} into rational case (Section 3.2) and transcendental case (Section 3.3).

\subsection{The first step of Theorem \ref{characterization-general}---rational case}\

We first prove Theorem \ref{characterization-general} for a rational function $f$  with degree $d\geq 2$.
\begin{proof}

 (1)First, we assume that $E(f)=\emptyset$. Since $\mathcal{J}(f)$ is the closure of repelling periodic points,  there exists a repelling periodic point $z_0$ in
$D(z,r)$ with period $m$.  We can find a small neighbourhood $U$ of $z_0$ in $D(z,r)$ such that $U \subset f^m(U)$.
Since $E(f)=\emptyset$, we know that
\[\bigcup_{n=1}^{+\infty}f^{nm}(U)=\hat{\mathbb{C}}.\]
By the compactness argument, we have
\[\bigcup_{n=1}^{k}f^{mn}(U)=\hat{\mathbb{C}}\]
for some $k.$ Thus we get $f^{mk}(U)=\hat{\mathbb{C}},$
which shows that
\[ d^{n}\leq S(f^{n+mk}, D(z,r))\leq  d^{mk+n}.\]

(2)
Secondly, we assume  that $E(f)$ $\neq \emptyset$. It is known that $\sharp E(f)\leq 2$ and $E(f)\subset \mathcal{F}(f)$. Further more each point in $E(f)$ is an attracting periodic point.
We write $E(f)=\{p_1, p_2\}$, where $p_1\neq p_2$. The proof for the case of $E(f)$ consisting of only one point is almost the same. By Proposition \ref{Characterization-Fatou-Set}, we know that there exists $\delta>0$ such that  $\sup\limits_{n}S(f^{n}, D(p_1,\delta))\leq \Gamma$ and  $\sup\limits_{n}S(f^{n}, D(p_2,\delta))\leq \Gamma$
     for some constant $\Gamma>0$. Noting that $f$ is a ramified  covering map with degree $d$, we know that
     \[S(f^n, \hat{\mathbb{C}})= d^n\]
    for any $n\in \mathbb{N}$. Now, we obtain
     \[S\left(f^n, \hat{\mathbb{C}}\backslash (D(p_1,\delta)\cup D(p_2,\delta))\right)\geq  d^n-2\Gamma.\]
    For any $r>0$, by the blowing up property of the Julia set,  there exists $k\in \mathbb{N}$ such that  $\hat{\mathbb{C}}\backslash (D(p_1,\delta)\cup D(p_2,\delta))\subset f^k(D(z_0, r))$. Now by Proposition \ref{Area-Relation}, we get
    \[S(f^{n+k}, D(z_0,r))\geq S\left(f^n, \hat{\mathbb{C}}\backslash (D(p_1,\delta)\cup D(p_2,\delta))\right)\geq  d^n-2\Gamma,\]
    and this implies that
    $$\lim\limits_{n\rightarrow +\infty}\frac{\log S(f^n, D(z_0,r))}{n}=\log d.$$

\end{proof}
Together with Proposition \ref{Characterization-Fatou-Set}, we have proved Theorem \ref{characterization-general} for rational case.

\subsection{The second step of Theorem 1--- transcendental case}\

In this subsection we prove Theorem \ref{characterization-general} for the transcendental case.
We need the following argument coming from Ahlfors covering surface theory.
\begin{lemma}[\cite{Hayman-Book}]\label{Ahlfors-covering}
Let $f$ be a transcendental entire (meromorphic) function on $\mathbb{C}$. Assume $\{K_i\}_{i=1}^{5}$ are five simply  connected Jordan domains such that $\overline{K_i}\cap \overline{K_j}=\emptyset$ for distinct $i$ and $j$. Then there exists some $K_i$ such that $f^{-1}(K_i)$ has infinitely many simply connected and bounded components.
\end{lemma}

Now we are in the position to prove Theorem 1 for the transcendental case.
\begin{proof}
For any transcendental entire (meromorphic) function $f$, $\sharp E(f)\leq 2$ follows from Montel's theorem. We choose five Jordan domains $\{D_i\}_{i=1}^{5}$  on $\mathbb{C} $ such that $\overline{D_i}\cap \overline{D_j}\neq \emptyset$. According to Lemma \ref{Ahlfors-covering}, we can choose some $j\in \{1, 2, 3, 4, 5\}$, such that there exist infinitely many simply connected and bounded domains $\{U_k\}_{k=1}^{+\infty}$ over $D_{j}$.  We know that
$E(f)$ intersects with at most two domains of $\{U_k\}_{k=1}^{+\infty}$. We choose the domains that do not intersect with $E(f)$, and we relabel  the domains as $\{U_{k'}\}_{k'=1}^{+\infty}$. For any $z_0\in \mathcal{J}(f)$,  $ r>0$ and  $l>0$, there exists $N=N(z_0, r, l)>0$ such that $\bigcup\limits_{k'=1}^{l}U_{k'}\subset f^n (D(z_0,r))$ for any $n\geq N$. We get
\[S(f^{n+1}, D(z_0,r))\geq l\cdot S(\Id, D_{j}).\]
This implies that  $\lim\limits_{n\rightarrow +\infty}S(f^n, D(z_0,r))=+\infty.$
\end{proof}

\section{Proofs of Theorem \ref{characterization-one} and Theorem \ref{characterization-E(f)-empty}}\label{Section-Theorem-2-3}
\subsection{A motivated example--Exponential map} \

In order to illustrate our method for the general case, we first discuss the example of exponential map.
\begin{Prop}
Let $f(z)=\exp(z)$. Then for any
 $ z\in \mathbb{R}$, we have
\begin{equation}
\lim\limits_{n\rightarrow +\infty}\frac{\log S(f^n, \ D(z,r))}{n}=+\infty
\end{equation}
 for any $r>0$. Specially, $\mathbb{R}\subset \mathcal{J}(f)$.
\end{Prop}

\begin{Remark}
 $\mathcal{J}(\exp)=\mathbb{C}$ was conjectured by Fatou in 1926.  Misiurewicz \cite{Misiurewicz1981} verified it in 1981. Some other proofs can also be found in \cite{Eremenko-Lyubich, Shen-Rempe-Gillen}. The key idea in their proofs can be summarized as two steps.
 \begin{enumerate}
 \item Prove that $\mathbb{R}\subset \mathcal{J}(\exp)$.
 \item Prove that for any $z\in \mathbb{C}$ and any neighborhood of $U$ of $z$, $\exp^n(U)$ will intersect $\mathbb{R}$ for some $n$.
 \end{enumerate}

  Here, we only give a pure computational proof for $\bigcup_{n=1}^{+\infty}\exp^{-n}(\mathbb{R})\subset \mathcal{J}(\exp)$.  It is interesting to ask whether there exists a pure computational proof for $\mathcal{J}(\exp)=\mathbb{C}$.
\end{Remark}
\begin{proof}
First for any $z_0\in \mathbb{R}_{+}=\{z: z>0\}$,
$D(z_0,r)$ lies also in the right half plane for any $r< \frac{1}{4} z_0$. Let $g_n(z)=f^n(z_0+z),$ we get $M(r,g_n)=\exp^n(z_0 +r)$.

Now we have
\begin{align*}
S(g_n, D(0,r))&=\frac{1}{\pi}\int\int_{D(0,r)}\frac{|g_n'(z)|^2}{(1+|g_n(z)|^2)^2}dxdy\\
&=\frac{1}{\pi}\int\int_{D(0,r)}\frac{|(f^n)'(z_0+z)|^2}
{(1+|f^n(z_0+z)|^2)^2}dxdy\\
&=\frac{1}{\pi}\int\int_{D(z_0,r)}\frac{|(f^n)'(z)|^2}
{(1+|f^n(z)|^2)^2}dxdy\\
&=S(f^n, D(z_0,r)).
\end{align*}
On the other hand, one has the following estimation.
\begin{align}
 &2\log 2\times S(g_n, D(0,4r))\notag\\
\geq &\int_{r}^{4r}\frac{S(t,g_n)}{t}dt=T_{0}(4r,g_n)-T_{0}(r,g_n)
\notag\\
\geq & (T(4r,g_n)-\frac{1}{2}\log2-\log^{+}|g_n(0)|)
-(T(r,g_n)+\frac{1}{2}\log2-\log^{+}|g_n(0)|)\label{inequality-exp-1}\\
= &T(4r,g_n)-T(r,g_n)-\log 2\notag\\
\geq & \frac{1}{3}\log^{+}M(2r, g_n)-\log^{+}M(r,g_n)-\log 2\label{inequality-exp-2}.
\end{align}
The  inequality \eqref{inequality-exp-1} above follows from Lemma \ref{lemma-Nevanlinna-Ahlfors-Shimizu} and the inequality
\eqref{inequality-exp-2} deduces from Lemma \ref{Modulus-Nevanlinna-Characteristic}.
 Then we get $$\lim_{n\rightarrow +\infty}\frac{\log S(f^n, D(z_0,r))}{n}=+\infty$$ by standard estimation. For the case that $z_0\leq 0$, we have $f(z_0)=\exp(z_0)>0$. Then for any $r>0$,
 there exists some $r_1>0$ such that $f(D(z_0, r)) \supseteq D(f(z_0), r_1)$. We know that
 $$\lim_{n\rightarrow +\infty}\frac{\log S(f^n, D(f(z_0),r_1))}{n}=+\infty$$ by the discussion above. Hence,
 $S(f^{n+1}, D(z_0,r))\geq S(f^n, D(f(z_0), r_1))$ and
 $$\lim_{n\rightarrow +\infty}\frac{\log S(f^n, D(z_0,r))}{n}=+\infty.$$
By Theorem \ref{characterization-general}, this implies that $\mathbb{R}$ lies in the Julia set of $\exp z.$
\end{proof}

In the following, we establish two lemmas related to the growth of maximal modulus, which are crucial to Theorem \ref{characterization-one} and Theorem \ref{characterization-E(f)-empty}.
\subsection{Two basic Lemmas}
 The following lemma follows a certain Hadamard convexity property.

\begin{lemma}\label{compareInequ}
 Let $f$ be a transcendental entire function. For any $k>0$, there exists $r=r(f)>0$, such that for any $r_2> r_1\geq r$, we have
\[\frac{M(r_2, f)}{M(r_1, f)}\geq \frac{r_2^k}{r_1^k}.\]
\end{lemma}
\begin{proof}
Noting that $f$ is transcendental, we have
\[\lim_{r\rightarrow +\infty}\frac{\log M(r,f)}{\log r}=+\infty.\]
For any $k>0$, there exists $\hat{r}>0$,
\[\frac{\log M(r,f)-\log M(1,f)}{\log r-\log 1}>k\]
for any $r>\hat{r}.$ Since $\log M(r,f)$ is a convex function with respect to $\log r$, this
implies that
\[\frac{\log M(r_2, f)-\log M(r_1, f)}{\log r_2-\log r_1}\geq \frac{\log M(r,f)-\log M(1,f)}{\log r-\log 1}>k \]
for any $r_2>r_1\geq r.$
\end{proof}

We also need the following lemma, which needs more careful estimations compared to Lemma \ref{compareInequ}. For convenience, we denote $M^2(r,f)=M(M(r,f),f)$ and
$M^n(r,f)=M(M^{n-1}(r,f),f)$ for general $n\in \mathbb{N}$ in the following discussion.

\begin{lemma}
 Let $f$ be a transcendental entire  function $f$. Given two points $z_0, z_1\in \mathbb{C}$ . Then there exists $R(f)>0$, if  $|z_1|> |z_0|\geq R(f)$,  we have
\[\lim_{n\rightarrow+\infty}\frac{1}{n}\log \log \frac{M^{n}( |z_1|, f)}{M^n( |z_0|, f)}=+\infty.\]
\end{lemma}
\begin{proof}
Since $f$ is a transcendental entire function, there exists $R>0$, if $|z_1|> |z_0|\geq R$, we have
\[\lim_{n\rightarrow+\infty}M^{n}(|z_i|,f)=+\infty\]
for $i=0, 1.$ For any $k$ in the statement in Lemma \ref{compareInequ},  there exist $r>0$, $L>0$ such that
\[M^L(|z_1|,f)>M^L(|z_0|,f)\geq r.\]
Writting $\frac{M^L(|z_1|, f)}{M^L(|z_0|, f)}=1+\epsilon$ for some $\epsilon>0$.
By Lemma \ref{compareInequ}, we have
\[\frac{M^{L+1}(|z_1|,f)}{M^{L+1}(|z_0|,f)}
=\frac{M(M^L(|z_1|, f),f)}
{M(M^L(|z_0|, f),f)}\geq
\left(\frac{M^{L}(|z_1|,f)}{M^{L}(|z_0|,f)}\right)^k.\]
Similarly,
\[\frac{M^{L+2}(|z_1|,f)}{M^{L+2}(|z_0|,f)}
=\frac{M(M^{L+1}(|z_1|,f),f)}
{M(|f^{L+1}(|z_0|,f)|,f)}\geq
\left(\frac{M^{L+1}(|z_1|,f)}{|M^{L+1}(|z_0|,f)}\right)^k\geq \left(\frac{M^{L}(|z_1|,f)}{M^{L}(|z_0|,f)}\right)^{k^2}.\]
Inductively,  we have
\[\frac{M^{n+L}(|z_1|,f)}{M^{n+L}(|z_0|,f)}\geq \left(\frac{M^{L}(|z_1|,f)}{M^{L}(|z_0|,f)}\right)^{k^n}\geq (1+\epsilon)^{k^n}.\]
Hence,
\[\liminf_{n\rightarrow+\infty}\frac{1}{n}\log \log \frac{M^{n+L}( |z_1|, f)}{M^{n+L}( |z_0|, f)}\geq \lim_{n\rightarrow +\infty}\frac{n}{n+L}\log k=\log k.\]

Noting that $k$ is arbitrary, we get
\[\lim_{n\rightarrow +\infty}\frac{1}{n}\log \log \frac{M( |z_1|, f^{n})}{M( |z_0|, f^{n})}=+\infty.\]
\end{proof}

We are now ready to prove Theorem 2.
\subsection{The proof of Theorem \ref{characterization-one}}\

We now prove the necessary part for Theorem \ref{characterization-one}. The main idea  is to estimate $S(f^n, U)$ by  Nevanlinna theory in terms of Ahlfors-Shimizu's characteristic, where $U$ is a large annuli such that $U\cap E(f)=\emptyset$.

\begin{proof}
For any transcendental entire funciton $g$, we rewrite
\[S(r, g)=\frac{1}{\pi}\int_{0}^{r}\rho d\rho\int_{0}^{2\pi}\frac{|g'(\rho e^{i\theta})|^2}{(1+|g(\rho e^{i\theta})|^2)^2}d\theta\]
for $r>0$.
And also for convenience,  we define
\[S(r_1, r_2, g)=\frac{1}{\pi}\int_{r_1}^{r_2}\rho d\rho\int_{0}^{2\pi}\frac{|g'(\rho e^{i\theta})|^2}{(1+|g(\rho e^{i\theta})|^2)^2}d\theta\]
for $0\leq r_1<r_2$. It is obvious that $S(r_1,r_2, g)=S(r_2,g)-S(r_1,g).$ We denote $D(z, r, R)=\{w:r< |z-w|< R\}$.

For any $1<d_1<\frac{d}{2}$, we can choose $k$ large enough such that
 $\frac{2^k+\frac{6}{5}}{2^k-\frac{6}{5}}<d_1$.

One has the following estimation
\begin{align}
\ & S(\frac{r}{2^{k+1}}, 2^{k-1}r, f^n)=S(2^{k-1}r,f^n)-S(\frac{r}{2^{k+1}},f^n)\notag\\
\geq & \frac{1}{k\log 2}\int_{\frac{r}{2}}^{2^{k-1}r}\frac{S(t)}{t}dt-\frac{1}{k\log 2}\int_{\frac{r}{2^{k+1}}}^{\frac{r}{2}}\frac{S(t)}{t}dt \notag\\
=&\frac{1}{k\log 2}(T_{0}(2^{k-1}r,f^n)-T_{0}(\frac{r}{2},f^n))-
\frac{1}{k\log 2}(T_{0}(\frac{r}{2},f^n)-T_{0}(\frac{r}{2^{k+1}},f^n))
\notag\\
\geq & \frac{1}{k\log 2}\left\{T(2^{k-1}r,f^{n})-T(\frac{r}{2},f^n)-
T(\frac{r}{2},f^n)+T(\frac{r}{2^{k+1}},f^n)\right\}-
\frac{2}{k} \label{inequality-theorem-2-2}\\
\geq &\frac{1}{k\log 2}\left(T(2^{k-1}r,f^n)-2T(\frac{r}{2},f^n)\right)-\frac{2}{k}.\notag
\end{align}
The inequality \eqref{inequality-theorem-2-2} follows by Lemma \ref{lemma-Nevanlinna-Ahlfors-Shimizu}.
Write $I_{n}(r)=S(\frac{r}{2^{2k-1}},2r, f^{n})\times k\log 2$, we have
\begin{align*}
I_{n}(r)&\geq \frac{2^{k-1}-\frac{6}{5}}{2^{k-1}+\frac{6}{5}}\log ^{+}M(\frac{6r}{5},f^n)-2\log^{+}M(\frac{r}{2},f^n)-2
\log 2\\
&\geq \frac{1}{d_1}\log ^{+}M(\frac{6r}{5},f^n)-2\log^{+}M(\frac{r}{2},f^n)-2
\log 2.
\end{align*}

To prove the theorem,  we need the following lemma.
\begin{lemma}[BRS\cite{BRS}]\label{Eremenko-Point}
Let $f$ be a transcendental entire function. Then for any $\epsilon>0$, there exists sufficiently large $r>0$ such that there is a point
\[w\in \{z: r<|z|<(1+\epsilon)r\}\]
with $|f^n(w)|\geq M^n(r,f)$ for any $n\in \mathbb{N}.$

\end{lemma}

 The point $w$ in Lemma \ref{Eremenko-Point} is called Eremenko point by Rippon and Stallard \cite{RipponStallard-EremenkoPoint}. It was first constructed by Eremenko \cite{Eremenko} with Wiman-Valiron theory.  The traditional proof of Wiman-Valiron theory is based on power series, which can not be extended to meromorphic functions. Bergweiler, Rippon and Stallard \cite{BRS-direct-tract} established a similar Wiman-Valiron theory for mermorophic functions with direct tracts by some delicate estimations of harmonic measures.

 We continue the proof of Theorem \ref{characterization-one}. By Lemma \ref{Eremenko-Point}, there exists $R_1>0$ such that if $r>R_1$, then there exists a point
\[w\in\{z:\frac{11}{10}r\leq |z|\leq \frac{6}{5}r  \}  \]
with $$|f^{n}(w)|\geq M^{n}(\frac{11}{10}r,f)$$
and $$M(r,f)>100 r$$
 for
all $n\in \mathbb{N}$.

Therefore, we have
\begin{align*}
I_{n}(r)&\geq
\frac{1}{d_1}\log^{+}|f^{n}(w)|-2
\log^{+}M^{n}(\frac{r}{2}, f)-2\log 2 \\
&\geq \frac{1}{d_1}\log M^{n}(\frac{11}{10}r,f)-2
\log M^{n}(\frac{r}{2}, f)-2\log 2\\
&=\frac{1}{d_1}\log \frac{M^{n}(\frac{11}{10}r,f)}
{(M^{n}(\frac{r}{2}, f))^{2d_1}}-2\log 2
\end{align*}
for $r\geq R_1.$ Let $r$ be fixed. By Lemma \ref{compareInequ}, for any $0<\epsilon <\frac{1}{100}$, there exists $N=N(r, \epsilon)$ such that
\[2M^{n}(\frac{r}{2},f)<M^{n}(\frac{1+\epsilon}{2}r,f)\]
for any $n\geq N.$ Now, by the growth assumption \eqref{Growth-Condtion}, we get
\[\left(M^{n}(\frac{r}{2},f)\right)^{2d_1}<M^{n}(\frac{1+\epsilon}{2}r,f).\]

 Since $\sharp E(f)\leq 2$, we have $D(0,\frac{r}{2^{k+1}},2^{k-1}r)\bigcap E(f)=\emptyset$
 for any $r>R_1$. $\forall z\in \mathcal{J}(f),  s>0$,
there exists $L=L(z, r)>0$ such that
\[D(0,\frac{r}{2^{k+1}},2^{k-1}r)\subset f^{L}(D(z,s)).\]

Thus,
\[S(f^{n+L},D(z,s))\geq S(f^n, D(0, \frac{r}{2^{k+1}}, 2^{k-1}r))\]
for all $n\in \mathbb{N}.$ This implies the following inequality
\begin{equation}
S(f^{n+L},D(z,s))
\geq \frac{1}{d_1k\log 2}\log \frac{M^{n}(\frac{11}{10}r,f)}
{M^{n}(\frac{1+\epsilon}{2}r, f)}-\log 2.
\end{equation}
In view of  Lemma \ref{compareInequ}, we obtain
\[\lim_{n \rightarrow +\infty}\frac{\log S(f^n, D(z,s))}{n}=+\infty.\]

\end{proof}
\subsection{Proof of Corollary \ref{characterization-k} }\

Now we prove Corollary \ref{characterization-k} by Theorem \ref{characterization-one}.

\begin{proof}

By Theorem \ref{characterization-one},   $z\in \mathcal{J}(f^k)$ if and only if
\[
\lim_{n\rightarrow +\infty}\frac{1}{n}\log S(f^{kn}, D(z,r))=+\infty
\]
for any $r>0$.
Noting that $\mathcal{J}(f^k)=\mathcal{J}(f)$, for any $z\in \mathcal{J}(f)$, we have $f^{i}(z)\in \mathcal{J}(f^k)$ for $1\leq i\leq k-1$. For any fixed $i$,  there exists $r_i>0$ such that $f^i(D(z,r))\subset D(f^i(z), r_i)$. Thus we have
$$S(f^{kn+i}, D(z,r))\supset S(f^{kn}, D(f^i(z), r_i)).$$
Therefore
$$
\lim_{n\rightarrow+\infty}\frac{\log S(f^{kn+i}, D(z,r))}{n}\geq \lim_{n\rightarrow+\infty}\frac{\log S(f^{kn}, D(f^i(z),r_i))}{n}=+\infty.
$$

\end{proof}

Now, we turn to prove Theorem \ref{characterization-E(f)-empty}.
The idea in Theorem \ref{characterization-E(f)-empty}
follows from the same steps in Theorem
\ref{characterization-one}.  Under the setting
of Theorem \ref{characterization-E(f)-empty},
 the key ingredient
here is to use the blowing up property to cover a
large open disk $U$, and we need to  estimate the lower
bound of $S(f^n, U)$. However, this strategy fails for Theorem
\ref{characterization-one}, and we use blowing up property to cover a large
annulus.
\subsection{Proof of Theorem \ref{characterization-E(f)-empty}}\

\begin{proof}

For any $1<d_1<d$, we can choose $k$ large enough such that
 $\frac{2^k+\frac{6}{5}}{2^k-\frac{6}{5}}<d_1$.

Write $I_{n}(r)=(k+1)\log 2\times S(2^k r, f^{n})$, and we
have the following estimation
\begin{align*}
I_{n}(r)\geq &
T(2^k r, f^{n})-T(\frac{r}{2}, f^{n})-\log 2\\
 \geq & \frac{2^k-\frac{6}{5}}{2^k+\frac{6}{5}}\log^{+}M(\frac{6}{5}r,f^{n})-
\log^{+}M(\frac{r}{2}, f^{n})-\log 2 \\
\geq & \frac{1}{d_1}\log^{+}M(\frac{6}{5}r,f^{n})-
\log^{+}M(\frac{r}{2}, f^{n})-\log 2
\end{align*}
Similarly as before,   there exist $R$ and a point
\[w\in\{z:\frac{11}{10}r\leq |z|\leq \frac{6}{5}r  \}  \]
such that $$|f^{n}(w)|\geq M^{n}(\frac{11}{10}r,f)$$
and $$M(r,f)>100 r$$
for
all $n\in \mathbb{N}$ and $r>R$.

Now let $r$ be fixed,  we have
\begin{equation}\label{inequality-theorem-2-1}
I_{n}(r)\geq
\frac{1}{d_1}\log^{+}|f^{n}(w)|-
\log^{+}M^{n}(\frac{r}{2}, f)-\log 2
\end{equation}

For any $0<\epsilon <\frac{1}{100}$, by Lemma \ref{compareInequ}, there exists $N=N(r, \epsilon)$ such that
\[2M^{n}(\frac{r}{2},f)<M^{n}(\frac{1+\epsilon}{2}r,f)\]
for any  $n\geq N$. Together with growth condition \eqref{Growth-Condtion-2}, we have
\[\left(M^{n}(\frac{r}{2},f)\right)^{d_1}<M^{n}(\frac{1+\epsilon}{2}r,f)\]
for all $n \geq N.$ $\forall z\in \mathcal{J}(f), s>0,$  since $E(f)=\emptyset$, there exists $L=L(z, s, r)>0,$ such that
\[D(0,2^k r)\subset f^{L}(D(z,s )).\]
 Therefore,
\[S(f^{n+L},D(z,s))\geq S(f^n, D(0,2^k r))=\frac{I_n(r)}{(k+1)\log 2}\]
for all $n\in \mathbb{N}.$ Hence, by inequality \eqref{inequality-theorem-2-1}, we deduce that
\[S(f^{n+L},D(z,s))\geq\frac{1}{(k+1)d_1\log 2}\log \frac{M^{n}(\frac{11}{10}r,f)}
{M^{n}(\frac{1+\epsilon}{2}r, f)}-\log 2.
\]
By Lemma \ref{compareInequ}, we obtain
\[\lim_{n \rightarrow\infty}\frac{\log S(f^n, D(z,s))}{n}=+\infty.\]
\end{proof}

\section{Proofs of Theorem   \ref{Multiply} and Theorem \ref{characterization-hyperbolic}}\label{Section-Theorem-4-5}

In the previous discussions, all results are limited to
the case that $\mathcal{J}(f)$ is connected. Now, we establish
some  results for the case $\mathcal{J}(f)$ is not connected.
We have seen that one needs to handle the case $E(f)\neq
\emptyset$ carefully. It is natural to think if we
 restrict some conditions on $E(f)$, we can still get  the
 characterization of Julia set.  We need  a  lemma related
 to multiply connected Fatou components
 and $E(f)$ in the proof of Theorem \ref{Multiply}.

\begin{lemma}\label{Multiply-Exceptional}
Let $f$ be a transcendental entire function such that
$\mathcal{F}(f)$ has one multiply connected Fatou component. Then $E(f)=\emptyset$.
\end{lemma}

 This lemma should be known, however, we can not find a direct reference. Hence we add a proof here.

\begin{proof}
Assume that $E(f)=\{0\}$, it is known that $f$ can be written as the following form:
\[f(z)=z^k\exp(g(z))\]
for some $k \in \mathbb{N}$ and entire function $g$.

By the  Borel-Caratheodory theorem,  for any transcendental entire function $f$,  we have
\[M(r, f)\leq \frac{2r}{R-r}A(R, f)\]
for $r<R$, where $A(R, f)=\max_{|z|=R}\Re f(z)$.
We consider the following inequality
\begin{align*}
\frac{\log^{+}M(4r,f)}{\log^{+}M(r,f)}&=\frac{k\log 4r+A(4r,g)}{k\log r+A(r,g)}\\
&\geq \frac{k\log 4r+\frac{5}{6}M(\frac{3r}{2},g)}{k\log r+M(r,g)}.
\end{align*}
The last inequality follows by  Lemma \ref{Modulus-Nevanlinna-Characteristic}.
Either $g$ is a polynomial or transcendental entire function, we have
$\frac{\log^{+}M(4r,f)}{\log^{+}M(r,f)}>d>1$ for large $r$.
 This implies that $f$ has mo multiply connected Fatou component, which is a contradiction.

\end{proof}
\begin{Remark}
One can also use Nevanlinna theory to prove this lemma. Any value $a\in E(f)$ must be a Nevanlinna deficient value, that is to say $\delta(a,f)>0$.
Since for any $a\in E(f)$, the pre-image of $a$ must be finite. We have $\delta(a, f)=1$.
Since $f$ has multiply connected Fatou component, we can choose a sequence of $\{r_n\}_{n=1}^{\infty}$ such that $$\lim_{n\rightarrow\infty}\min\limits_{|z|=r_n}|f(z)|=\infty.$$ Now
we have
\[\liminf_{r\rightarrow\infty}m(r, \frac{1}{f-a})=\lim_{r\rightarrow\infty}\int_{0}^{2\pi}\log^{+}
\frac{1}{|f(re^{i\theta})-a|}d\theta=0.\]
This directly implies that $\delta(a, f)=0$, which is a contradiction.
\end{Remark}

Now, we turn to prove Theorem \ref{Multiply}.

{\bf Proof of Theorem \ref{Multiply}}
\begin{proof}
Let $U$ be  a multiply connected Fatou component with eventual connectivity $2$. Denote $U_n=f^n(U).$ There exists $n_0$ such that the connectivity of $U_n$ is equal to $2$ for all $n\geq n_0.$ In this case, Bergweiler, Rippon and Stallard \cite{BRS} have proved that there exists no critical point  in $\bigcup\limits_{n=n_0}^{\infty}U_n.$  This means that $f:U_n\rightarrow U_{n+1}$ is a  finite degree unramified covering map and we assume that the degree is $d_n$.  There exist a sequence of geometric annulus $\{A_{r_n}\}_{n=n_0}^{\infty}$ and maps $\{h_n\}_{n=n_0}^{\infty}$ and $\{g_n\}_{n=n_0}^{\infty}$ such that the following diagram commutes
\[\begin{array}{cccccc}
U_{n} &
\stackrel{f}{\longrightarrow} &
U_{n+1} & \stackrel{f}{\longrightarrow} &
U_{n+2} & {\longrightarrow\cdots}\\
\Big\downarrow \vcenter{%
\rlap{$\scriptstyle{h_{n}}$}} & &
\Big\downarrow\vcenter{%
\rlap{$\scriptstyle{h_{n+1}}$}}& &
\Big\downarrow \vcenter{%
\rlap{$\scriptstyle{h_{n+2}}$}} \\
A_{r_n} & \stackrel{g_n}{\longrightarrow} &
A_{r_{n+1}} & \stackrel{g_{n+1}}{\longrightarrow} &
A_{r_{n+2}}& {\longrightarrow\cdots} .\\
\end{array}\]

In the diagram,  $h_{n}$ is the bi-holomorphic map from $U_{n}$ to $A_{r_n}$, where $$A_{r_n}=\{z: 1\leq |z| \leq \exp(2 \pi \Mod(U_n))\}.$$

And also $g_n$ is a ramified covering map with degree $d_n$. By Schwarz reflection principle, we know that $g_n(z)=a_nz^{d_n}$ for some $a_n\neq 0$. This means that the image of any circle $\{z: |z|=m_n\}$ for some $m_n$ in $A_{r_n}$ is always a geometric circle in $A_{r_{n+1}}$.

One can choose an analytic Jordan curve $\gamma_n\subset h_n^{-1}(A_{r_n})=U_n$ which is not homotopic to $0$.
Let $\gamma_{n+1}=f(\gamma_n).$  We know that $\gamma_{n+1}$ is a analytic Jordan curve.
It is known that
$f:\gamma_n\rightarrow \gamma_{n+1}$ is a covering map with degree $d_n$. Since $f$ is an entire function,
$f:\intt \gamma_n\rightarrow \intt \gamma_{n+1}$ is a ramified covering map with the same degree $d_n$.  Furthermore, we get $\lim\limits_{n\rightarrow\infty}d_n=\infty$ by argument principle.

By Lemma \ref{Multiply-Exceptional}, $E(f)=\emptyset$. For any $z\in \mathcal{J}(f)$ and any $r>0$, there exists $L>0$, such that $\intt\gamma_{n_0}\subset f^{L}(D(z,r))$. Noting that $f:\intt \gamma_{n_0}\rightarrow \intt \gamma_{n_0+1}$ is a ramified covering map with degree $d_{n_0}$, one gets
$S(f, \intt\gamma_{n_0})=d_{n_0}\cdot S(\Id, \intt\gamma_{n_0+1})$. Therefore,
\[S(f^{L+1}, D(z,r))\geq S(f, f^{L}(D(z,r)))\geq S(f, \intt\gamma_{n_0})=d_{n_0}\cdot S(\Id, \intt\gamma_{n_0+1}).\]
Similarly, we obtain
$$S(f^{n+L}, D(z,r))\geq d_{n_0}d_{n_0+1}\dots d_{n_0+n-1}\cdot S(\Id, \intt \gamma_{n_0}).$$

This indicates
\[
\liminf_{n\rightarrow+\infty}\frac{\log S(f^{n+L}, D(z,r))}{n}\geq
\liminf_{n\rightarrow+\infty}\frac{\log S(\Id, \intt\gamma_{n_0})+\sum_{j=0}^{n-1}\log d_{n_0+j}}{n}=+\infty.
\]

\end{proof}
{\bf Proof of Theorem \ref{characterization-hyperbolic}}
\begin{proof}
First, we assume that $E(f)\neq \emptyset$ and $0\in E(f)$.  Choose small $r_1>0$, such that $D(0,2r_1) $ still lies in some  Fatou component. We know that $f^n(D(0, 2r_1))$ converges to one attracting periodic orbit when $n$ goes to infinity. This implies that there exists some constant $\Gamma>0$ such that $\sup\limits_{n}\log^{+} M(2r_1, f^n)\leq \Gamma$. This leads to
\begin{equation}\label{hyperbolic-attracting}
\sup\limits_{n} T(2r_1, f^n)\leq \Gamma
\end{equation}
by Lemma \ref{Modulus-Nevanlinna-Characteristic}.

By Lemma \ref{compareInequ}, there exists $R>0$, such that
\[\lim_{n\rightarrow+\infty}\frac{1}{n}\log \log M^n(r_2,f)=+\infty\]
for $r_2>R$. Now let $r_1, r_2$ be fixed and $r_2>100 r_1$, then we have

\begin{align*}
\ & S(r_1, 2r_2, f^n)=S(2r_2,f^n)-S(r_1,f^n)\\
\geq & \frac{1}{\log \
\frac{2r_2}{r_1}}\int_{r_1}^{2r_2}\frac{S(t)}{t}dt-\frac{1}{\log 2}\int_{r_1}^{2r_1}\frac{S(t)}{t}dt\\
=& \frac{1}{\log \frac{2r_2}{r_1}}(T_{0}(2r_2,f^n)-T_{0}(r_1,f^n))-
\frac{1}{\log 2}(T_{0}(2r_1,f^n)-T_{0}(r_1,f^n))
\\
\geq & \frac{1}{\log \frac{2r_2}{r_1} }(T(2r_2,f^n)-T(r_1,f^n))-\frac{1}{\log 2}(T(2r_1,f^n)-T(r_1,f^n))-1-\frac{\log 2}{\log \frac{2r_2}{r_1} }\\
\geq & \frac{1}{\log \frac{2r_2}{r_1} }T(2r_2,f^n)-B
\end{align*}
for some constant $B$ depending on $r_1$ and $r_2$ and $\Gamma$.
The last inequality follows from \eqref{hyperbolic-attracting}.
Furthermore, by the similar technique in Theorem \ref{characterization-one}, we know that there exists some constant $A$, such that
\[S(r_1, 2r_2, f^n)\geq  A\log^{+}M^n(r_2,f)-B\]
by Lemma \ref{Modulus-Nevanlinna-Characteristic} and Lemma \ref{Eremenko-Point}.
Hence we get
\[\lim_{n\rightarrow+\infty}\frac{\log S(r_1, 2r_2, f^n) }{n}=+\infty.\]

Now for any $ z\in \mathcal{J}(f)$ and any $r>0$,
there exists $L=L(z,r, r_1, r_2)>0,$ such that
\[D(0, r_1, 2r_2)\subset f^{L}(D(z,r)).\]

Now we have
\[\liminf_{n\rightarrow +\infty}\frac{\log S(D(z,r), f^{n+L}) }{n}\geq \liminf_{n\rightarrow+\infty}\frac{S(f^n, D(0, r_1, 2r_2))}{n}=+\infty.\]

For the case $E(f)=\emptyset$, we can assume that $0$ lies in some Fatou component. The proof is the same.

\end{proof}
\section*{Acknowledgments} We would like to thank  professors Jianhua Zheng and Guangyuan Zhang for enjoyable  discussions  on Ahlfors covering surface theory and Nevanlinna theory. We also thank professor Guizhen Cui who draws our attention to other metrics in the characterization.  Last, we thank Yinying Kong and Fujie Chai's great help during the preparation of this manuscript.

\bibliographystyle{amsplain}

\end{document}